\documentclass[12pt,english]{article}
\usepackage[latin9]{inputenc}
\usepackage{amssymb}
\usepackage{amsthm}
\PassOptionsToPackage{normalem}{ulem}
\usepackage{ulem}

\newtheorem{thm}{Theorem}
\newtheorem{prop}[thm]{Proposition}
\newtheorem{lemma}[thm]{Lemma}
\newtheorem{claim}[thm]{Claim}
\newtheorem*{unclaim}{Claim}
\newtheorem{cor}[thm]{Corollary}

\theoremstyle{definition}
\newtheorem{defi}[thm]{Definition}
\newtheorem{rem}[thm]{Remark}

\newcommand{\R}{\mathbb{R}}
\def \Rplus {[0,\infty)}

\newcommand{\Rn}{\mathbb{R}^n}
\newcommand{\mz}{\setminus\sz}
\newcommand{\sz}{\cbrackets{0}}

\def \Rplus {[0,\infty)}

\def \intRn {\int_{\Rn}}

\def \intSn {\int_{\Sn}}
\def \intOn {\int_{O_n}}
\def \Sn {S^{n-1}}
\def \Snzero {S^{n-2}_0}
\def \E {\mathbb{E}}
\def \P {\mathbb{P}}

\def \bigo {\mathcal{O}}

\def \grad {\nabla}
\def \tab {\ \ \ \ }
\def \pushleft {\tab\tab\tab\tab\tab\tab\tab\tab\tab\tab\tab\tab\tab\tab\tab\tab\tab\tab\tab\tab\tab\tab\tab\tab\tab\tab\tab\tab\tab}

\def \I {\mathds{1}}

\def \sleq {\lesssim}
\def \sgeq {\gtrsim}

\def \nline {\hspace{1mm}\\}

\newcommand{\inlinesection}[1]{\textbf{{#1}.}}

\def \Tr {\text{Tr}}
\def \Pthperp {P_{\theta^\perp}}
\def \sign {\sigma_{n-1}}
\def \Var {\text{Var}}

\newcommand{\eref}[1]{(\ref{#1})}

\def \Id {Id}
\def \Loneiso {$L^1$-isotropic}
\def \Lpiso {$L^p$-isotropic}
\def \Lpcov {$L^p$-covariance}
\def \Covp {\Cov_p}

\def \dcube {\cbrackets{-1,1}^n}

\def \Cov {\text{Cov}}

\def \On {O_n}
\newcommand{\oo}[1]{\frac{1}{#1}}
\newcommand{\half}{\oo{2}}

\newcommand{\cbrackets}[1]{\left\{{#1}\right\}}
\newcommand{\norm}[1]{\|{#1}\|}
\newcommand{\normtext}[2]{\norm{#1}_{\text{#2}}}
\newcommand{\opnorm}[1]{\normtext{#1}{op}}
\newcommand{\hsnorm}[1]{\normtext{#1}{HS}}

\newcommand{\cb}[1]{\cbrackets{#1}}
\newcommand{\icol}[1]{
	\left(\begin{smallmatrix}#1\end{smallmatrix}\right)%
}

\usepackage{amsmath}
\usepackage{babel}
\usepackage{enumitem}
\usepackage{dsfont}
\usepackage[usenames, dvipsnames]{color}
\usepackage{tikz}
\usetikzlibrary{patterns}
\begin{document}

\setcounter{tocdepth}{1}
\setcounter{page}{1}
\newpage

\title{Concentration between Lévy's inequality and the Poincaré inequality for log-concave densities}
\date{}
\author{Erez Buchweitz\thanks{School of Mathematical Sciences, Tel Aviv University, Tel Aviv 69978, Israel. Email: erezmb@gmail.com. Supported in part by the European Research Council. This article is based on the author's M.Sc. thesis, prepared under the supervision of Prof. Bo'az Klartag.}}

\maketitle

\inlinesection{Abstract} Given a suitably normalized random vector $X\in\Rn$ we observe that the function $\theta\mapsto\E|X\cdot\theta|$, defined for $\theta\in\Sn$, admits surprisingly strong concentration far surpassing what is expected on account of Lévy's isoperimetric inequality. Among the measures to which the above holds are all log-concave measures, for which a solution of the similar problem concerning the third marginal moments $\theta\mapsto\E (X\cdot \theta)^3$ would imply the hyperplane conjecture.\\

\section{Introduction}

The aim of this note is twofold. We expand on a remark made by R. Eldan and B. Klartag concerning the hyperplane conjecture in convex geometry, while commenting on the confines of Lévy's isoperimetric inequality for functions on the sphere. In their paper \cite{EK} which establishes the connection between the thin shell property and the hyperplane conjecture, Eldan and Klartag observe that the hyperplane conjecture would be affirmed if a dimension-free upper bound is established for the quantity
\begin{equation}\label{third.moment.variance}
n^2\intSn \big(\E(X\cdot\theta)^3\big)^2 \ d\sign(\theta),
\end{equation}
valid for all isotropic log-concave random vectors $X\in\Rn$. Here, we denote by $\sign$ the uniform probability measure over the unit sphere $\Sn=\{x\in\Rn:x_1^2+...+x_n^2=1\}$ and by $X\cdot \theta$ the Euclidean inner product. Stated differently, the integral above is the variance of the function $\theta\mapsto \E(X\cdot\theta)^3$ taken with respect to $\sign$.
\nline

Though no insight is offered on how to achieve the bound \eref{third.moment.variance} on the third moment variance, they do however point to a similar problem concerning \textit{first} marginal moments. They suggest it might hold that
\begin{equation*}
\Var_\theta \ \E_X|X\cdot\theta| \ \leq \ C/n^2,
\end{equation*}
whenever $X$ is log-concave, symmetric and suitably normalized, with $C>0$ a universal constant. It is on this problem that we wish to elaborate.
\nline

We undertake here a slightly different perspective. Given a centered Borel probability measure $\mu$ on $\Rn$ with finite first moment, define a function on the sphere by
\begin{equation}\label{Fdef}
F_\mu(\theta) \ = \ \intRn (x\cdot\theta)_+ \ d\mu(x),
\end{equation}
where $t_+=\max\cb{t,0}$. Notice that $F_\mu(\theta)=F_\mu(-\theta)=\E_X|X\cdot\theta|/2$ whenever $X$ is a random vector distributed according to $\mu$. Geometrically, $F_\mu(\theta)$ is the $\theta$-component of the (unnormalized) center of mass of $\mu$ on the half-space in the direction of $\theta$.
In order to study $F_\mu$, we assume a normalization of $\mu$ in which the matrix
\begin{equation*}
\Cov_1(\mu) \ = \ \intRn x\otimes x \ \frac{d\mu(x)}{|x|},
\end{equation*}

whose elements are $\intRn x_ix_j/|x| \ d\mu(x)$, $i,j=1,...,n$ is scalar, with $|\cdot|$ denoting the Euclidean norm. Our first result reads as follows.
\begin{thm}\label{thm1}
	Let $\mu$ be a centered Borel probability measure on $\Rn$. Assume that for some $\alpha,\beta\in(0,\infty)$, $\Cov_1(\mu)=\alpha/\sqrt{n}\cdot\Id$ and 
	\begin{equation}\label{iint4}
	\intRn\intRn \frac{(x\cdot y)^4}{|x|^3|y|^3} \ d\mu(x)\ d\mu(y) \ \leq \ \frac{\beta\alpha^2}{n}.
	\end{equation}
	Then,
	$$
	\Var(F_\mu) \ \leq \ \frac{C(1+\beta)\alpha^2}{n^2}
	$$
	where $C>0$ is a universal constant
\end{thm}
We say that a probability measure $\mu$ is \textit{centered} if it has a \textit{center of mass} at the origin, i.e. $\intRn x\ d\mu=0$. When dealing with concentration of functions on the sphere, the classical result is the isoperimetric inequality due to Paul Lévy (see \cite{G, L}), relating concentration with the magnitude of the spherical gradient. Whenever $f:\Sn\to\R$ is a Lipschitz function, Lévy's inequality implies the bound
\begin{equation}\label{isoperimetric}
\|f\|_{L^{\psi_2}(\sign)} \ \leq \ \frac{C}{\sqrt{n-1}} \ \sup_{\theta\in\Sn } |\grad_S f(\theta)|
\end{equation}
and in particular $\Var(f) \leq 4\sup |\grad_S f|/(n-1)$. We generally use the symbols $c,C,C',\tilde{C}$ and so on to denote universal constants whose values differ between occurrences. Consider as well the spherical Poincaré inequality which provides the seemingly superior variance bound
\begin{equation}\label{poincare}
\tab\tab\ \ \ \Var(f) \ \leq \ \oo{n-1} \intSn |\grad_S f(\theta)|^2 \ d\sign(\theta).
\end{equation}
The effect of ``super concentration" discussed by Chatterjee \cite{Ch} is the non-tightness on inequality \eref{poincare}. In our case, inequality \eref{poincare} is essentially tight but much stronger than inequality \eref{isoperimetric}, $\E|\grad_S F_\mu|\ll \sup|\grad_S F_\mu|$. Even so we have a $\psi_1$ bound in the log-concave case (see Theorem \ref{thm5}). In particular $\E |\grad_S F_\mu|^2\leq C(1+\beta)\alpha^2/n$, while $\sup|\grad_S F_\mu|$ cannot be assumed in general to be anything less than a universal constant times $\alpha$, as is demonstrated for example by the uniform measure on the discrete cube.
The evidence we bring here relates to our geometric family of functions $F_\mu$ alone, yet we believe this strong concentration is a manifestation of some deep far-reaching phenomenon in concentration of measure.
Other related effects are discussed by Bobkov, Chistyakov and Götze \cite{BCG} and by Paouris and Valettas \cite{PV}.
\nline

The normalization where $\Cov_1(\mu)$ is scalar differs from the isotropic one, which requires the covariance matrix to be the identity. It may be though that both $\Cov_1(\mu)$ and the covariance matrix are scalar, for instance when a measure is the joint distribution of even, independent, identically distributed random variables. Regarding assumption \eref{iint4}, we view it as a regularity condition which indeed holds true in many scenarios, and we provide examples in the following pages. The argument we use to obtain Theorem \ref{thm1} may be furthermore applied to third marginal moments in an effort to bound the quantity \eref{third.moment.variance}, by employing the $L^3$-isotropic normalization as defined in Section 4. It seems this approach can only yield a reduction of \eref{third.moment.variance} to familiar problems known to imply the hyperplane conjecture. 
 \nline

It is worth noting that assumption \eref{iint4} cannot be dropped, as evidenced by the example of the discrete measure distributed evenly among the vectors of an orthonormal basis and their negatives. Before going on to give examples of cases in which the conditions of Theorem \ref{thm1} hold, we describe an additional assumption under which an even tighter concentration occurs.
\begin{thm}\label{thm2}
	In the setting of Theorem \ref{thm1}, if the assumptions hold with $\alpha>0$, 
	\begin{equation*}\label{beta3}
	\beta\leq 3+\gamma/n
	\end{equation*} and moreover $\iint (x\cdot y)^6/|x|^5|y|^5  d\mu(x) d\mu(y) \leq  \delta\alpha^2/n^2$	for some $\gamma, \delta >0$, then 
	\begin{equation*}
	\Var(F_\mu) \ \leq \ \frac{C(1+\gamma+\delta)\alpha^2}{n^3}
	\end{equation*}
	where $C>0$ is a universal constant.
\end{thm}
Inequality \eref{iint4} with $\beta=3$ applies to the standard Gaussian probability measure, as well as any other spherically symmetric measure. It also holds, for example, for the discrete measure evenly distributed among the vertices of the discrete cube. We suspect the assumptions mentioned in Theorem \ref{thm2} hold at least for all sufficiently regular unconditional measures, i.e. when the density is invariant under reflection with respect to any of the axes.
\nline

We continue with the example of the discrete cube.
\begin{thm}\label{thm3}
	Let $\nu_n$ be the probability measure distributed uniformly on the discrete cube $\dcube$. Then 
	$$
	\Var(F_{\nu_n}) \leq \ C/n^3,
	$$
	where $C>0$ is a universal constant.
\end{thm} 

This result is asymptotically optimal. A particular case of Khinchine's inequality states that summing real numbers $a=(a_1,...,a_n)$ with random signs yields the tight bounds 
\begin{equation*}
\frac{|a|}{\sqrt{2}} \ \leq \ \E\Big|\sum_{i=1}^n \varepsilon_i a_i\Big| \ \leq \ |a|,
\end{equation*}
where $\varepsilon_1, ...,\varepsilon_n$ are independent random signs. Even though the values of $\E |(\varepsilon_1,...,\varepsilon_n)\cdot a|$ range  in $[1/\sqrt{2},1]$ for $|a|=1$ in any dimension, Theorem \ref{thm3} implies that the variance over $\Sn$ diminishes at a fast rate of $1/n^3$. 
\nline

We may also consider random subsets of the discrete cube.
\begin{thm}\label{thm4}
	Let $X_1,...,X_N$ be random vertices of the discrete cube $\dcube$ taken independently, $N=n^{2+\delta}$ for some $\delta>0$. Define $\mu$ to be the discrete probability measure evenly distributed among $X_1,...,X_N$ (with repetitions). 
	
	Then with probability at least $1-\gamma(n, \delta)$ of choosing $X_1,...,X_N$,
	$$
	\Var(F_\mu) \ \leq \ C/n^2,
	$$
	where $\gamma(n,\delta) = 2n^2\exp\{-n^\delta/2\}+\exp\{-2n^{1+\delta/2}\}$
	and $C>0$ is a universal constant.
\end{thm}
Note that $\gamma(n,\delta)\to 0$ as $n\to\infty$, when $\delta>0$ is fixed. It is plausible that a more delicate analysis will lead to a stronger $C/n^3$ variance bound, as with the measure supported on the entire discrete cube. This would perhaps require taking a somewhat larger subset.
\nline

Moving on, for log-concave measures an elegant result can be stated, in the form of a $\psi_1$ bound.
\begin{thm}\label{thm5}
	For any absolutely-continuous log-concave probability measure $\mu$ on $\Rn$, $n\geq C''$, there exists an affine position (namely, the \Loneiso\ position) in which
	$$
	\|F_\mu\|_{L^{\psi_1}(\sign)} \ \leq \ C/n.
	$$
	for all $t>0$. In particular, we obtain the moment bounds
	$$
	\|F_\mu-\E F_\mu\|_{L^p(\sign)}^p \ = \ \intSn |F_\mu(\theta)-\E F_\mu|^p\ d\sign(\theta) \ \leq \ (C'p/n)^p,
	$$
	for any $p\geq 1$. Here, $c,C,C'>0$ are universal constants.
\end{thm}

In the above $\E F_\mu$ is the mean of $F_\mu$ and $\P=\sign$.
Note that for $p=2$ we get $\Var(F_\mu) \leq C/n^2$. A few clarifications are in order. First, recall that any log-concave measure is either absolutely continuous (a.c. for short) or has a density on some lower-dimensional subspace. Hence the assumption that the measure is a.c. is reasonable. Second, by affine position (or image) we mean  the pushed-forward measure under an invertible affine transformation. Third, the position in question, which we call the \Loneiso\ position, involves the matrix $\Cov_1(\mu)$ being scalar. A precise, generalized definition is given in Section 4. To avoid trivialities we note that assuming the \Loneiso\ normalization we have $\E F_\mu\geq c$ where $c>0$ is a universal constant. Finally, we may again wish to compare the concentration bounds given in Theorem \ref{thm5} to Lévy's inequality, which implies the $\psi_2$ bound \eref{isoperimetric} and the corresponding moment bounds
$$
\|F_\mu-\E F_\mu\|_{L^p(\sign)} \ \leq \ C\sqrt{p}/\sqrt{n}.
$$
We have thus improved the concentration implied by Lévy's inequality by a factor of $\sqrt{p}/\sqrt{n}$ for all moments. It is not immediately clear how to extend the $C/n^3$ variance bound from Theorem \ref{thm2} into an exponential tail bound stronger than that of Theorem \ref{thm5}.
\nline

The rest of this note is structured as follows. In Section 2 we lay out the general terms for a tighter variance bound and prove Theorems \ref{thm1} and \ref{thm2}. In Section 3 we discuss the example of the discrete cube and prove Theroems \ref{thm3} and \ref{thm4}. Section 4 is devoted to studying the \Loneiso\ position and its application to log-concave measures. This will help toward the proof of Theorem \ref{thm5}, which appears in the final Section 5.
\nline

Some proofs are omitted from the body of this note, and appear in full in the Appendix or in the author's M.Sc. thesis \cite{thesis}.
\nline 

\inlinesection{Acknowledgements} I hold a great deal of gratitude to Prof. Bo'az Klartag, my thesis supervisor, without whose guidance and insight this work could not have materialized.

\section{Conditions for tight variance}

We first provide a direct proof of Theorem \ref{thm2}.

\begin{proof}[Proof of Theorem \ref{thm2}]
Write $\Var(F_\mu) \ = \ \E F_\mu ^2 \ - (\E F_\mu)^2$ and evaluate each expression separately. By the definition \eref{Fdef} of $F_\mu$ and by rearranging the order of integration, the two components may be written as
\begin{equation}\label{var.decomp1}
\E F_\mu^2 \ = \ \intRn\intRn |x||y| \bigg(\intSn \Big(\theta\cdot \frac{x}{|x|}\Big)_+\Big(\theta\cdot \frac{y}{|y|}\Big)_+ d\sigma(\theta) \bigg) d\mu(x) \ d\mu(y)
\end{equation}
and 
\begin{equation}\label{var.decomp2}
\E F_\mu \ = \ \intRn |x|\bigg(\intSn \Big(\theta\cdot\frac{x}{|x|}\Big)_+  d\sigma(\theta)\bigg) d\mu(x).\tab\tab\tab\tab\tab\tab\tab\ 
\end{equation}
The two inner integrals over the sphere appearing in equations \eref{var.decomp1} and \eref{var.decomp2} may be approached in a similar manner. We demonstrate the argument for the inner integral of equation \eref{var.decomp1}; reduce the high-dimensional integral to expose the fact that the solution depends only on the angle between $x$ and $y$. Whenever $f:\Rn\to\R$ is a $p$-homogeneous function, i.e. $f(tx)=t^pf(x)$ for all $x\in\Rn$, $t>0$, polar integration yields the following change of variable formula for the standard Gaussian probability measure $\gamma_n$ on $\Rn$
\begin{equation}\label{polar}
\intRn f(x) \ d\gamma_n(x) \ = \ C_{n,p}\intSn f(\theta) \ d\sign(\theta),
\end{equation}
with the constant $C_{n,p}=n2^{p/2-1}\Gamma((n+p)/2)/\Gamma((n+2)/2)$, where $\Gamma$ is the Gamma function (for a proof, see \cite[Chapter 2]{thesis}). Write $x=|x|\eta$ and $y=|y|\xi$ for $\eta,\xi\in\Sn$ and apply the change of variable \eref{polar} to our 2-homogeneous function $\theta\mapsto (\theta\cdot\eta)_+(\theta\cdot\xi)_+$ of equation \eref{var.decomp1}. We get
\begin{align*}
\intSn (\theta\cdot\eta)_+(\theta\cdot\xi)_+ d\sigma(\theta) \ &= \ C_{n,2}^{-1}\intRn (z\cdot\eta)_+(z\cdot\xi)_+\ d\gamma_n(z)
\\&= \ C_{n,2}^{-1}\int_{\R^2} (z\cdot\eta)_+(z\cdot\xi)_+\ d\gamma_2(z)
\\&= \ \frac{C_{n,2}^{-1}}{2\pi} \int_0^{2\pi}\big(\icol{\cos \varphi\\\sin \varphi} \cdot \icol{1\\0}\big)_+\big(\icol{\cos \varphi\\\sin \varphi}\cdot\icol{\cos\rho\\\sin\rho}\big)_+ \ d\varphi
\\&= \ \frac{C_{n,2}^{-1}}{2\pi} \int_0^{2\pi} \cos(\varphi)_+\cos(\varphi-\rho)_+ \ d\varphi
\\&= \ \oo{2\pi n}\big((\pi-\rho)\cos\rho + \sin\rho\big),
\end{align*}
where we have denoted $\cos\rho = \eta\cdot\xi = (x\cdot y)/(|x||y|)$, $\rho\in[0,\pi]$ and as $C_{n,2}=n$. In the above, we projected the $n$-dimensional integral onto the plane containing $\eta$ and $\xi$ (permitting they are not colinear), then applied the change of variable formula \eref{polar} in reverse direction on the 2-dimensional plane.\\

Continuing, expand the function $\varphi(\tau) = (\pi-\arccos\tau)\tau + \sqrt{1-\tau^2}$ where $\arccos\tau\in[0,\pi]$ into a power series around 0, getting $\varphi(\tau) = 1 + \pi\tau/2+\tau^2/2+\tau^4/24+\bigo(\tau^6)$ where the notation $\bigo(\tau^6)$ represents a quantity whose absolute value is at most a universal constant times $\tau^6$. For a proof of this expansion, see \cite[Chapter 2]{thesis}. Applying the expansion to $\tau=\cos\rho$ and plugging into the computation \eref{var.decomp1} we arrive at
\begin{align*}
\E F_\mu^2 \ = \ \iint_{\Rn\times\Rn} \bigg(\frac{|x||y|}{2\pi n} \ + \ &\frac{x\cdot y}{4n} \ + \ \frac{(x\cdot y)^2}{4\pi n |x||y|} \
\\
+ \ &\frac{(x\cdot y)^4}{48\pi n|x|^3|y|^3} \ + \ \bigo\bigg(\frac{(x\cdot y)^6}{n|x|^5|y|^5}\bigg)\bigg) \ d\mu\otimes\mu.
\end{align*}
In a similar fashion the expression in equation \eref{var.decomp2} may be reduced to
\begin{equation*}
(\E F_\mu)^2 \ = \ \iint_{\Rn\times\Rn} \bigg( \frac{|x||y|}{2\pi n} \ + \ \frac{|x||y|}{4\pi n^2} \ + \ \frac{|x||y|}{16\pi n^3} \ + \  \bigo\bigg(\frac{|x||y|}{n^4}\bigg)\bigg) \ d\mu\otimes \mu,
\end{equation*}
having approximated $C_{n,1}^{-2}=1/n+1/2n^2+1/8n^3+\bigo(1/n^4)$ (see \cite{TE}). At this point note that the component linear in $x\cdot y$ equals after integration to the magnitude of the center of mass of $\mu$ squared and thus vanishes as $\mu$ is centered. Moreover,
\begin{equation*}
\int |x| d\mu \ = \ \Tr\Cov_1(\mu) \ = \ \alpha\sqrt{n} \tab\text{and}\ \  \int (x\cdot y)^2 \ \frac{d\mu(x)}{|x|} \ = \ \Cov_1(\mu)y\cdot y \ = \ \frac{\alpha|y|^2}{\sqrt{n}},
\end{equation*}
hence subtracting $(\E F_\mu)^2$ from $\E F_\mu^2$ we arrive at
\begin{equation*}
\Var(F_\mu) \ = \ \oo{48\pi n}\bigg(\iint \frac{(x\cdot y)^4}{|x|^3|y|^3} \ d\mu\otimes \mu \ - \ \frac{3\alpha^2}{n} \bigg)\ + \  \bigo\bigg(\frac{\alpha^2}{n^3} \ + \ \frac{\delta\alpha^2}{n^3}\bigg),
\end{equation*}
and the proof of the Theorem is completed by applying the final assumption. 
\end{proof}
To prove Theorem \ref{thm1}, simply repeat the argument above while taking the two series expansions up to the orders $\bigo(\tau^4)$ and $\bigo(1/n^3)$, respectively.
\nline

When $\mu$ is absolutely continuous $F_\mu$ is spherically differentiable, even $C^1$-smooth (see \cite[Chapter 2]{thesis}). To describe its gradient, we introduce the notation $\Pthperp:\Rn\to\Rn$ for the orthogonal projection onto the hyperplane perpendicular to $\theta$, and $H_\theta = \cb{x\in\Rn: x\cdot\theta>0}$ the half-space through the origin in the direction of $\theta$. Then the Euclidean gradient when we extend $F_\mu$ 1-homogeneously to $\Rn\mz$ is given by
\begin{equation}\label{2ddef}
\grad F_\mu(\theta) \ = \ \int_{H_{\theta}} x \ d\mu(x)
\end{equation}
and its spherical counterpart is $\grad_S F_\mu(\theta) = \Pthperp \grad F_\mu(\theta)$. We refer to \cite{BCG} for an introduction of the notion of the spherical derivative.
The quantity $\E |\grad_S F_\mu|^2$ we are able to compute along the lines of the proof of Theorem \ref{thm2}.
\begin{prop}\label{gradprop}
	Under the assumptions of Theorem \ref{thm1} and the additional assumption that $\mu$ is absolutely continuous,
	$$
	\intSn |\grad_S F_\mu(\theta)|^2 \ d\sign(\theta) \ \leq \ \frac{C(1+\beta)\alpha^2}{n}.
	$$
	\nline	
	Under the assumptions of Theorem \ref{thm2} and the additional assumption that $\mu$ is absolutely continuous,
	\begin{equation*}
	\intSn |\grad_S F_\mu(\theta)|^2 \ d\sign(\theta) \ \leq \ \frac{C'(1+\gamma+\delta)\alpha^2}{n^2}.
	\end{equation*}
	
	In the above, $C, C'>0$ are universal constants.
	
\end{prop}

The proof of Proposition \ref{gradprop} appears in Appendix \ref{app.gradprop}. We may thus obtain Theorems \ref{thm1} and \ref{thm2} readily from Proposition \ref{gradprop} via the spherical Poincaré inequality
\begin{equation*}
\Var(f) \ \leq \ \E |\grad_S f|^2\ /\ (n-1)
\end{equation*}
pertaining to all $C^1$-smooth functions of the sphere. Proposition \ref{gradprop} will be of further use in succeeding discussions.
\nline

Yet a third way to obtain Theorems \ref{thm1} and \ref{thm2} is via the second-order Poincaré inequality on the sphere, put forth recently by S. G. Bobkov, G. P. Chistyakov and F. Götze \cite{BCG}; whenever $f:\Sn\to\R$ is $C^2$-smooth and has no linear spherical harmonic component,
\begin{equation*}
\Var(f) \ \leq \ \E \hsnorm{f''_S}^2 \ / \ 2n(n+2).
\end{equation*}
Indeed, $F_\mu$ is even and under some regularity conditions $F_\mu$ is twice differentiable and $\E\ \|(F_\mu)_S''\|_{HS}^2$ can be bounded accordingly, see Appendix \ref{app.2deriv}.
\nline

\section{The discrete cube}

As an instructive example, we consider the case of the discrete cube $\dcube$. It is straightforward to show that the discrete measure assigning equal probability to each of the cube's vertices, formally defined as
\begin{equation*}
\nu_n \ = \ \sum_{x \in \dcube} \delta_x/2^n
\end{equation*} 
where $\delta_x$ denotes an atom of mass one at $x$, meets the requirements of Theorem \ref{thm2} with $\alpha=1$, $\gamma=0$, $\delta=15$ in any dimension.

\begin{proof}[Proof of Theorem \ref{thm3}]
The measure $\nu_n$ is centered, and taking $X$ to be a random vector distributed according to $\nu_n$ we see that
\begin{equation*}
\sqrt{n}\Cov_1(\nu_n)_{ij} \ = \ \E X_iX_j \ = \ \delta_{ij}
\end{equation*}
as the coordinates of $X$ are independent random signs, where $\delta_{ij}$ here is Kronecker's delta. Adding $Y$ a second random vertex independent of $X$ and distributed according to $\nu_n$,  simple combinatorial calculations verify the remaining assumptions of Theorem \ref{thm2}. First,
\begin{align*}
n^3\iint \frac{(x\cdot y)^4}{|x|^3|y|^3} \ d\nu_n\otimes\nu_n \ &= \ \E_Y \E_X (X\cdot Y)^4 \pushleft
\\&= \ \E\Big(\sum_{i=1}^n X_i\Big)^4 
\\&= \ \sum_{i=1,..,n} \E X_i^4 \ + \ 3\cdot \sum_{\underset{i\neq j}{i,j=1,...,n}} \E X_i^2X_j^2 \ = \ 3n^2-2n.
\end{align*}
Accordingly,
\begin{align*}
n^5\iint &\frac{(x\cdot y)^6}{|x|^5|y|^5} \ d\nu_n\otimes\nu_n \ = \ \E \Big(\sum_{i=1}^n X_i\Big)^6
\\&= \ \sum_{i=1,..,n} \E X_i^6 \ + \ 15\cdot \sum_{\underset{i\neq j}{i,j=1,...,n}} \E X_i^2X_j^4 \ + \ 15\cdot \sum_{\underset{i\neq j\neq k}{i,j,k=1,...,n}} \E X_i^2X_j^2X_k^2
\\&= \ n \ + \ 15n(n-1) \ + \ 15n(n-1)(n-2) \ = \ 15n^3 \ + \ \bigo(n^2)\pushleft
\end{align*}

and this concludes the proof.
\end{proof} 
\begin{rem}
	The Lipschitz semi-norm of $F_{\nu_n}$ is at least $1/\sqrt{8}$ in any dimension, as may be illustrated by computing the magnitude of the spherical gradient at the points $\theta=(\cos\varphi, \sin\varphi, 0,...,0)$ for $\varphi\in(0, \pi/4)$. See \cite[Chapter 3]{thesis}.
\end{rem}
We now turn to proving Theorem \ref{thm4}. In the proof we clearly assume that $n\geq 8$. As we are dealing with a measure supported on a \textit{random} subset of the discrete cube, it would be unreasonable to expect the measure to have, for example, a center of mass \textit{exactly} at the origin, as required by Theorem \ref{thm1}. A delicate review of the proof of Theorem 1 will reveal that the Theorem's assumptions may be relaxed to the extent that they need to be met just approximately.
\begin{prop}[Alternative to Theorem \ref{thm1}]\label{prop6}
	Let $\mu$ be a Borel probability measure on $\Rn$. Assume that for some $\beta, \kappa, \lambda, \zeta \in (0,\infty)$,
	\begin{enumerate}
		\item $\big|\intRn x \ d\mu \big| \ \leq \ \kappa/\sqrt{n}$,
		\item $\big|n\cdot\|\Cov_1(\mu)\|_{HS}^2 - \text{Tr}\Cov_1(\mu)^2\big| \ \leq\  \lambda$,
		\item $\Tr\Cov_1(\mu)\ \leq\  \zeta\sqrt{n}$,
		\item $\iint \frac{(x\cdot y)^4}{|x|^3|y|^3} \ d\mu\otimes \mu \ \leq \ \beta/n$.
	\end{enumerate}
	Then, 
	\begin{equation*}
	\Var(F_\mu) \ \leq \ \frac{C(\kappa^2 + \lambda + \zeta^2 + \beta)}{n^2}
	\end{equation*}
	where $C>0$ is a universal constant.
\end{prop}
	
A proof of Proposition \ref{prop6} is given in \cite[Chapter 3]{thesis}.
We adopt the following setting. Let $X_1,...,X_N$ be i.i.d random vectors distributed according to $\nu_n$ with $N$ as in Theorem \ref{thm4},f and define 
\begin{equation*}
\mu \ = \ \sum_{i=1}^N \delta_{X_i}/N
\end{equation*}
a discrete probability measure on a subset of $\dcube$. It is important to note that the $N\cdot n$ coordinates of $X_1,...,X_N$ constitute independent random signs. Whenever $\varepsilon_1,...,\varepsilon_K$ are independent random signs, a concentration bound of
\begin{equation*}
\P \bigg(\oo{K}\Big|\sum_{k=1}^K \varepsilon_k\Big|>t\bigg) \ \leq \ 2e^{-Kt^2/2},
\end{equation*}
for all $t>0$ can be obtained via the Chernoff method (see e.g. \cite[Section 2.2]{BLM}).

\begin{proof}[Proof of Theorem \ref{thm4}]
	We verify that $\mu$ meets the requirements of Proposition \ref{prop6}. First, it holds with probability at least $1-2ne^{-N/2n^2}$ that \begin{equation*}
		\bigg|\intRn x\ d\mu\bigg|\ \leq \ 1/\sqrt{n}.
		\end{equation*}
	Indeed, the center of mass is $\sum_{i=1}^N X_i/N$ and we have
	\begin{equation*}
	\P\bigg(\left|\intRn x \ d\mu\right|^2>1/n\bigg) \ \leq \ \sum_{j=1}^n \P\bigg(\bigg|\sum_{i=1}^N X_{ij}\bigg| > N/n\bigg) \ \leq \ 2ne^{N/2n^2}.
	\end{equation*}
	Second, the diagonal elements of $\Cov_1(\mu)$ are all $1/\sqrt{n}$,
	\begin{equation*}
	\tab\tab\Cov_1(\mu)_{jj} \ = \ \intRn x_j^2 \ d\mu/|x| \ = \ 1/\sqrt{n}, \tab\ \ \ j=1,...,n
	\end{equation*}
	hence already $\Tr\Cov_1(\mu) = \sqrt{n}$, and furthermore the second requirement of Proposition \ref{prop6} boils down to the sum of the off-diagonal elements being small enough. Evidently with probability at least $1-2n(n-1)e^{-N/2n(n-1)}$ we have
	\begin{equation*}
	\sum_{\underset{j\neq k}{j,k=1,...,n}} \Cov_1(\mu)_{jk}^2 \ = \ \sum_{\underset{j\neq k}{j,k=1,...,n}} \bigg(\oo{N}\sum_{i=1}^n \frac{X_{ij}X_{ik}}{\sqrt{n}}\bigg)^2 \ \leq \ 1/n. 
	\end{equation*}
	Indeed,
	\begin{align*}
	\P\bigg(\sum_{j\neq k} \Cov_1(\mu)_{jk}^2 > \oo{n}\bigg) \ \leq \ \sum_{j\neq k}\P\bigg(\bigg|\sum_{i=1}^N X_{ij}X_{ik}\bigg|>N/\sqrt{n(n-1)}\bigg)
	\end{align*}
	and the claim follows as $\{X_{ij}X_{ik}\}_{i=1}^N$ is a set of independent random signs.
	
	Finally, we show with probability at least $1-e^{-2\sqrt{N}}$ that
	\begin{equation}\label{claim.schmuck}
	n\iint_{\Rn\times\Rn} \frac{(x\cdot y)^4}{|x|^3|y|^3} \ d\mu\otimes\mu \ = \ \oo{n^2} \intRn \bigg(\sum_{j=1}^n x_j\bigg)^4 d\mu \ \leq \ 2^{17}e^2.
	\end{equation}
	The proof employs a tail bound for a sum of iid $\psi_\alpha$ random variables described by M. Schmuckenschläger \cite{S}, and here we use the assumption that $n$ is large enough; if $Z$ is a non-negative random variable with $A:=\E \exp{\sqrt{Z}}<\infty$, then whenever $Z_1,...,Z_N$ are independent copies of $Z$ with $N\geq64/A$,
	\begin{equation}\label{schmuck}
	\P \bigg(\oo{N}\sum_{i=1}^N Z_i>t\bigg) \ \leq \ \exp{-\sqrt{Nt/2^5}}
	\end{equation}
	for all $t\geq 2^6\sqrt{A}$. In our case, set $Y_i \ = \ \big(\sum_{j=1}^n X_{ij}\big)^4/n^2$ and apply inequality \eref{schmuck} to $Y_1/2^{10}e^2$; we have
	\begin{align*}
	1 \ \leq \ \E \exp\big(\sqrt{Y_1}/2^5e\big) \ = \ \sum_{p=0}^\infty \frac{(2^5e)^{-p}}{p!}\E\ Y_1^{p/2} \ \leq \ \sum_{p=0}^\infty\frac{(2^5e)^{-p}}{p!} 2^{4p} p^p \ = \ 2.
	\end{align*}
	We used here the well-known Khinchine's inequality  for a sum of independent random signs (see  \cite{H}); $\E \big(\sum_{j=1}^n X_{1j}\big)^p\leq (8np)^{p/2}$. With application of inequality \eref{schmuck} the bound \eref{claim.schmuck} is obtained and the proof of Theorem \ref{thm4} is concluded.
\end{proof}

\section{The \Lpiso\ position}

We introduce the normalization that will be a primary tool in our treatment of the log-concave case.
\begin{defi}\label{lpcov}
	Let $\mu$ be a centered Borel probability measure on $\Rn$, and let $p>0$. Assume that $Z_{p,\mu}:=\intRn |x|^{p-2} d\mu$ is finite and nonzero. We say that $\mu$ is \textit{\Lpiso}\ if 
	\begin{equation*}
	\Covp(\mu)=Z_{p,\mu}\Id,
	\end{equation*} 
	where the \textit{\Lpcov\ matrix} is defined by
	\begin{equation*}
	\Covp(\mu) \ = \ \intRn x\otimes x \ |x|^{p-2} \ d\mu.
	\end{equation*}
\end{defi}

Notice that the $L^2$-covariance matrix is just the covariance matrix, and the $L^2$-isotropic normalization is the isotropic normalization. Together with the assertion that under mild assumptions any probability measure has an affine position in which it is \Loneiso, the main goal of this section will be to prove the following.
\begin{prop}\label{sect4lemma}
	Let $\mu$ be an \Loneiso\ log-concave probability measure on $\Rn$, with $n\geq C'$. Then $c \leq \sqrt{n}Z_{1,\mu}\leq C$ and
	\begin{equation*}
	\intRn\intRn \frac{(x\cdot y)^4}{|x|^3|y|^3} \ d\mu(x) \ d\mu(y) \ \leq \ C''/n,
	\end{equation*}
	where $c,C,C',C''>0$ are universal constants.
\end{prop}
An a.c. log-concave measure is one whose density is a log-concave function, i.e. of the form $e^{-H}$ for a convex function $H:\Rn\to(-\infty,\infty]$. We refer to  \cite{BGVV} for a comprehensive introduction of log-concave measures and their importance through their connection to convex bodies.
Proposition \ref{sect4lemma} suffices to establish that Theorem \ref{thm1} may be applied to \Loneiso\ log-concave measures, resulting in the bound
\begin{equation*}
\Var(F_\mu) \ \leq \ C/n^2
\end{equation*}
with the universal constant $C>0$ applying to all \Loneiso\ log-concave probability measures in any large enough dimension. 
\begin{rem}
	It is worthwhile noting that an \Loneiso\ log-concave measure has its first moment bounded between two constants; $\int |x| \ d\mu = \Tr \Cov_1(\mu) \ = \ \sqrt{n}Z_{1,\mu} \in (c,C)$. Hence an \Loneiso\ measure is non-degenerate in a sense.
\end{rem}

We now describe a few properties of the \Lpiso\ normalization. First, under mild integrability conditions a measure always has an affine position in which it is \Lpiso.
\begin{lemma}[Existence]\label{existence}
	Let $\mu$ be a centered Borel probability on $\Rn$, and let $p>0$. Assume that both $\int |x|^p \ d\mu$ and $\int |x|^{p-2} \ d\mu$ are finite and non-zero, and that the support of $\mu$ is not contained in any hyperplane. Then there exists an \Lpiso\ linear position of $\mu$.
\end{lemma}

Positions such as these are known to arise from extremal problems \cite{GM}. A fairly standard argument shows that the \Lpiso\ position results from minimizing the $L^p$ norm $\intRn |Sx|^p\ d\mu$ over all $S\in SL_n$. For a proof of Lemma \ref{existence}, see \cite[Chapter 4]{thesis}.
Next, a uniqueness property of the \Lpiso\ position. The group of orthogonal transformations is denoted as $O_n$.
\begin{lemma}[Uniqueness]\label{uniqueness}
	Let $\mu$ be an \Lpiso\ probability measure on $\Rn$ with $0<p<4$, and let $T\in GL_n$. Then  $T_*\mu$ is \Lpiso\ if and only if $T\in O_n$.
\end{lemma}

It is immediate that any orthogonal image of an \Lpiso\ measure is \Lpiso\ as well. The converse is proven using the following claim.
\begin{unclaim}\label{uniq.claim}
	Let $\mu$ be a Borel probability measure on $\Rn$ and let $0<p<4$. If both $\Covp(\mu)$ and $\Covp(S_*\mu)$ are scalar for $S\in SL_n$, then $S\in O_n$.
\end{unclaim}
\begin{proof}
	Apply an orthogonal transformation making $S$ positive definite, then diagonalize $S$ to have decreasing positive diagonal elements, and denote those by $\lambda_1\geq ...\geq \lambda_n>0$. Write $\Covp(\mu)=\kappa \Id$ and $\Covp(S_*\mu)=\kappa_S \Id$, and note that $\lambda_n|x| \leq |Sx|\leq \lambda_1|x|$ for all $x\in \Rn$. Assume $p\leq2$. Then
	\begin{equation*}
	\kappa_S \ = \ \Cov(S_*\mu)_{11} \ = \  \int x_1^2\ |x|^{p-2} \ dS_*\mu \ = \ \lambda_1^2\int x_1^2\ |Sx|^{p-2} \ d\mu\ \geq \ \lambda_1^p\kappa.
	\end{equation*}
	On the other hand,
	\begin{equation*}
	\ \kappa_S \ = \ \Cov(S_*\mu)_{nn} \ = \ \int x_n^2 \ |x|^{p-2} \ dS_*\mu \ = \ \lambda_n^2 \int x_n^2 \ |Sx|^{p-2} \ d\mu \ \leq \ \lambda_n^p\kappa.
	\end{equation*}
	Hence $\lambda_1^p\leq \lambda_n^p$. When $p>0$, this entails $\lambda_n\leq \lambda_1$ and so $S$ is scalar, hence $S$ is the identity matrix as $S\in SL_n$. The exact reverse argument is applicable to the case $2\leq p < 4$.
\end{proof}
Given two \Lpiso\ positions dilate one to use the Claim, then the normalization $\Cov_p(\mu)=Z_{p,\mu}\Id$ necessitates the dilation was trivial and Lemma \ref{uniqueness} is proven. We note that our approach here is inherently applicable to the range $0<p<4$ only. Extending the uniqueness property to $p>4$ would require a more delicate argument which eludes us, and we were not able to find in the literature. Here is simple a corollary of existence and uniqueness.
\begin{cor}\label{cor.uniqueness}
	Let $\mu$ be an \Lpiso\ probability measure on $\Rn$ with $0<p<4$. Then for any $S\in SL_n$ we have
	\begin{equation*}
	\intRn |x|^p \ d\mu \ \leq \ \intRn |x|^p \ dS_*\mu.
	\end{equation*}
\end{cor}
\nline
We conclude our discussion of the \Lpiso\ normalization with a Lemma relating the \Lpiso\ position to the isotropic one, in the case where the measure is log-concave. In a sense that is meaningful to us, the two positions are not ``too far" apart.
\begin{lemma}[Proximity to the isotropic position]\label{proximity}
	Let $\mu$ be an \Lpiso\ log-concave probability measure on $\Rn$, $n\geq C'$, and let $1\leq p \leq 2$. If $T\in GL_n$ is such that $T_*\mu$ is isotropic then
	\begin{equation*}
	\|T^{-1}\|_{\text{op}},\ \|T\|_{\text{op}} \ \in \ [c, C].
	\end{equation*}
	Here, $c,C,C'>0$ are universal constants.
\end{lemma}

The proof appears in Appendix \ref{app.proximity}, and is based on ``reverse Hölder"-type inequalities of the sort that are available for log-concave measures. Specifically, if $\mu$ is an a.c. log-concave probability measure on $\Rn$ and $f:\Rn\to\R$ is a semi-norm, then 
\begin{equation}\label{holder}
\bigg(\intRn |f|^q \ d\mu \bigg)^{1/q} \ \leq \ C\frac{q}{p}\bigg(\intRn |f|^p \ d\mu \bigg)^{1/q}
\end{equation}
whenever $1\leq p < q$. Similarly, an equivalence-of-moments result involving negative powers was proven by G. Paouris, according to which 
\begin{equation}\label{paouris}
\bigg(\intRn |x|^{-k} \ d\mu\bigg)^{1/k} \ \leq \ C_k\bigg(\intRn |x|^k \ d\mu\bigg)^{-1/k}
\end{equation}
whenever $\mu$ is an isotropic log-concave probability measure on $\Rn$ and $1\leq k \leq c\sqrt{n}$ is an integer (see \cite[Theorems 2.4.6 and 5.3.2]{BGVV}). It is important to note that as long as the values of $p,q,k$ used are bounded by some fixed value, inequalities \eref{holder}, \eref{paouris} can be formulated with constants that do not depend on $p,q,k$. In this case, combining with the reverse-order inequalities derived from Jensen's inequality, inequalities \eref{holder}, \eref{paouris} may be restated as 
\begin{equation*}
\bigg(\int |f|^q \ d\mu\bigg)^{\oo{q}} \ \simeq \ \bigg(\int |f|^p \ d\mu\bigg)^{\oo{p}} \tab;\tab \bigg(\int |x|^{-k} \ d\mu\bigg)^{\oo{k}} \ \simeq \ \bigg(\int |x|^k\ d\mu\bigg)^{-\oo{k}}
\end{equation*}
having adopted a notation where $A\simeq B$ means $cA\leq B\leq CA$ for some universal constants $c,C>0$. Similarly, we write $A\sleq B$ when $A\leq CB$. We are now in position to prove Proposition \ref{sect4lemma}.

\begin{proof}[Proof of Proposition \ref{sect4lemma}]
	Assume that $\mu$ is \Loneiso, and let $T\in GL_n$ be such that $T_*\mu$ is isotropic. By Lemma \ref{proximity} we know that $|Tx| \simeq |x| \simeq |T^{-1}x|$
	for all $x\in\Rn$. Therefore,
	\begin{equation*}
	nZ_{1,\mu} \ = \ \int |x| \ d\mu \ = \ \int |T^{-1}x| \ dT_*\mu \ \simeq \ \int |x| \ dT_*\mu \ \simeq \ \sqrt{n}
	\end{equation*} 
	since $T_*\mu$ is isotropic. As for the second assertion of the Lemma, 
	\begin{align*}
	\int \frac{(x\cdot y)^4}{|x|^3} \ d\mu(x) \ &\leq \ \sqrt{\int (x\cdot y)^8\ d\mu(x)}\cdot \sqrt{\int\frac{d\mu}{|x|^6}}
	\\&= \ \sqrt{\int (x\cdot T^{-1,*}y)^8 \ dT_*\mu(x)}\cdot\sqrt{\int\frac{dT_*\mu}{|T^{-1}x|^6}}
	\\& \simeq \bigg(\int (x\cdot T^{-1,*}y)^2 \ dT_*\mu(x)\bigg)^2\cdot\sqrt{\int\frac{dT_*\mu}{|x|^6}}
	\\&\simeq |T^{-1,*}y|^4/n^{3/2} \ \simeq \ |y|^4/n^{3/2}
	\end{align*}
	by applying inequalities \eref{holder} and \eref{paouris}. Finally,
	\begin{equation*}
	\iint \frac{(x\cdot y)^4}{|x|^3|y|^3} \ d\mu(x) \ d\mu(y) \ \sleq \ n^{-3/2}\int |y| \ d\mu \ \simeq \ n^{-1}.
	\end{equation*}
\end{proof}

\section{Log-concave measures}

We now wish to extend the variance bound obtained in Section 4 for log-concave measures, to the exponential tail bound described in Theorem \ref{thm5}. We obtain the tail bound via the equivalent bound on the  $\psi_1$ norm
\begin{equation}\label{expo}
\intSn e^{cn|F_\mu-\E F_\mu|} \ d\sign \ \leq \ 2
\end{equation}
with $c>0$ a universal constant. 
An argument for obtaining such bounds as \eref{expo} on the sphere was recently presented by Bobkov, Chistyakov and Götze \cite{BCG}, and requires controlling the second derivative. Namely, whenever $f:\Sn\to\R$ is a mean-zero $C^2$-smooth function with its spherical second derivative matrix admitting  
\begin{equation*}
\|f''_S(\theta)\|_{\text{op}} \ \leq \ 1
\end{equation*}
for any $\theta\in\Sn$, we have the exponential integral bound
\begin{equation}\label{bcg.exp}
\log\ \E\ e^{(n-1)f/2} \ \leq\ \E\ |\grad_S f|^2\cdot (n-1)/2.
\end{equation}
In the case that $f$ is defined on a neighborhood of the sphere, its spherical second derivative matrix may be defined by its relation to the Euclidean one
\begin{equation}\label{deriv2def}
f''_S(\theta) \ = \ \Pthperp(f''(\theta) - (\grad f(\theta)\cdot\theta)\Id)\Pthperp.
\end{equation}

Recall Proposition \ref{gradprop}, which together with Proposition \ref{lpcov} of Section 4 establishes for a log-concave measure in \Loneiso\ position the gradient bound
\begin{equation*}
\E\ |\grad_S F_\mu|^2 \ \leq \ C/n
\end{equation*}
with $C>0$ a universal constant. Combining this with inequality \eref{bcg.exp}, we will obtain the bound \eref{expo}, hence Theorem \ref{thm5}, once we prove the following.
\begin{prop}\label{propderiv2d}
	Let $\mu$ be an \Loneiso\ log-concave probability measure on $\Rn$, $n\geq C'$. Then $F_\mu$ is $C^2$-smooth and 
	\begin{equation*}
	\|(F_\mu)''_S(\theta)\|_{\text{op}} \ \leq C
	\end{equation*}
	for all $\theta\in\Sn$. Here $C,C'>0$ are universal constants.
\end{prop}

Indeed, when $\mu$ is log-concave $F_\mu$ is twice-differentiable when extended 1-homogeneously to $\Rn\mz$ and its second derivative matrix is given by
\begin{equation}\label{2dcalc}
F_\mu ''(\theta) \ = \ \int_{\theta^\perp} x\otimes x \ \rho(x) \ dx,
\end{equation}
where $\rho:\Rn\to[0,\infty)$ is the density of $\mu$ and the $(n-1)$-dimensional integral is taken over $\theta^\perp=\cb{x\in\Rn: x\cdot\theta=0}$, the hyperplane perpendicular to $\theta$. A proof of the differentiability of $F_\mu$ appears in \cite[Chapter 5]{thesis}. By the definition \eref{deriv2def} of the spherical second derivative, we have the bound
\begin{equation}\label{Fsop}
\|F''_S(\theta)\|_{\text{op}} \ \leq \ \|F''(\theta)|_{\theta^\perp}\|_{\text{op}} \ + \ |\grad F(\theta)\cdot\theta|.
\end{equation}

The right-hand summand is exactly $F(\theta)$, bounding which is immediate by the \Loneiso\ normalization. As to the left-hand summand, we will show that the problem reduces to the following two-dimensional assertion concerning sections of a log-concave function.

\begin{lemma}\label{lemmalogconc}
	Let $\omega:\R^2\to[0,\infty)$ be the density of an isotropic log-concave probability measure and let $\theta\in\R^2$, $|\theta|=1$. Then,
	\begin{equation*}
	\int_{-\infty}^\infty t^2\ \omega(t\theta) \ dt \ \leq C
	\end{equation*}
	where $C>0$ is a universal constant.
\end{lemma}

\begin{proof}
	By a reverse-Hölder-type inequality for one-dimensional log-concave functions \cite[Theorem 2.2.3]{BGVV} applied to $t\mapsto \omega(t\theta)$ we have
	\begin{equation*}
	\int_{-\infty}^\infty t^2 \ \omega(t\theta) \ dt \ \leq \ 2\omega(0)^{-2}\bigg(\int_{-\infty}^{\infty} \omega(t\theta) \ dt\bigg)^3.
	\end{equation*}
	It is well-known that among all two-dimensional isotropic log-concave probability measures, the density at the origin $\omega(0)$ is bounded from below by a universal constant \cite[Proposition 2.3.12]{BGVV}. Moreover, it holds true that
	\begin{equation*}
	\int_{-\infty}^\infty \omega(t\theta) \ dt \ \leq \ \sqrt{2}
	\end{equation*}
	and this will conclude the proof of the Lemma. Indeed, apply the reverse-Hölder-type inequality again, this time to the log-concave $\pi(t)=\int_{-\infty}^\infty \omega(t\eta+s\theta) \ ds$ where $|\eta|=1$, $\eta\in\theta^\perp$. Using as well that $\omega$ is isotropic we get
	\begin{equation*}
	1 \ = \ \int_{\R^2} (x\cdot\eta)^2 \ \omega(x) \ dx \ = \ \int_{-\infty}^\infty t^2 \ \pi(t) \ dt \ \leq \ \frac{2}{\pi(0)^2}\bigg(\int_{-\infty}^\infty \pi\bigg)^3 \ = \ \frac{2}{\pi(0)^2}
	\end{equation*}
	as required.
\end{proof}

\begin{proof}[Proof of Proposition \ref{propderiv2d}]
	By inequality \eref{Fsop} we must show that $|F_\mu''(\theta)\eta\cdot\eta|\leq C$ for all $\theta,\eta\in\Sn$, $\eta\perp\theta$.
	Fix such $\theta,\eta$. Restrict $F_\mu$ to the plane $E$ containing $\theta,\eta$, then $F_\mu|_E=F_{P_*\mu}$ where $P:\Rn\to E$ is the orthogonal projection. Indeed,
	\begin{equation*}
	\intRn (x\cdot\xi)_+ \ d\mu(x) \ = \ \int_E (x\cdot \xi)_+ \ dP_*\mu(x)
	\end{equation*}
	for all $\xi\in E$. Moreover, we have $F_\mu''(\theta)\eta\cdot\eta=F_{P_*\mu}''(\theta)\eta\cdot\eta$ as differentiating at $\theta$ twice in the direction $\eta$ incorporates evaluating $F_\mu$ only at points in $E$ (for a formal proof see \cite[Chapter 5]{thesis}). The measure $P_*\mu$ is a log-concave probability measure, though it is not necessarily \Loneiso. Nevertheless, it retains the property of proximity to the isotropic position. Namely, if $S\in GL_2$ is such that $S_*(P_*\mu)$ is isotropic, then $\opnorm{S},\opnorm{S^{-1}}$ are bounded from above by a universal constant. To see this, take $T\in GL_n$ such that $T_*\mu$ is isotropic and define $S=(PT^{-1}T^{-1,*}P^*)^{-1/2}\in GL_2$. Then $S_*(P_*\mu)$ is isotropic, and moreover $\opnorm{S},\opnorm{S^{-1}}\leq C$ because $cId_E \leq (T^{-1}T^{-1,*})|_E\leq CId_E$ in the sense of positive definite matrices, by Theorem \ref{proximity}.
	Denoting the density of $P_*\mu$ by $\omega_{P_*\mu}:\R^2\to[0,\infty)$ we may now calculate
	\begin{align*}
	F_{P_*\mu}''(\theta)\eta\cdot\eta \ = \ \int_{-\infty}^\infty t^2 \ \omega_{P_*\mu}(t\eta) \ dt \ = \ \int_{-\infty}^{\infty} t^2 \ \frac{\omega_{S_*(P_*\mu)}(tS\eta)}{|\det S^{-1}|} \ dt.
	\end{align*}
	Making the change of variable $t=s/|S\eta|$ and noticing that $|\det S^{-1}|,|S\eta|\simeq 1$ we get that the above equals up to a factor of a universal constant to
	\begin{align*}
	\int_{-\infty}^\infty s^2 \ \omega_{S_*(P_*\mu)}(s\xi) \ ds
	\end{align*}
	for $\xi=S\eta/|S\eta|$. The proof of the Proposition is concluded by applying Lemma \ref{lemmalogconc}.
\end{proof}

\newpage

\section{Appendix}

\subsection{}\label{app.gradprop}

\begin{proof}[Proof of Proposition \ref{gradprop}]
	We follow along the proof of Theorem \ref{thm2}. We first prove the second assertion of Proposition \ref{gradprop}. Write $|\grad_S F_\mu(\theta)|^2 \ = \ |\grad F_\mu(\theta)|^2 - (\grad F_\mu(\theta)\cdot\theta)^2$. By equality \eref{2ddef} and after rearranging the order of integration, we may write
	\begin{equation}
	\E |\grad F_\mu|^2 \ = \ \iint_{\Rn\times\Rn} (x\cdot y)\bigg(\intSn \I_{\theta \cdot x \geq 0}\ \I_{\theta\cdot y \geq 0} \ d\sign(\theta) \bigg)\ d\mu(x)\ d\mu(y)
	\end{equation}
	where the symbol $\I$ represents  the indicator function assuming value one if the condition is satisfied, zero otherwise. The inner integral over the sphere is the proportion of the sphere that lies in the intersection of two half-planes, and is simply $\big(\pi-\arccos \big((x\cdot y)/|x||y|\big)\big)/2\pi$.
	Alternatively, one may formally apply the polar integration formula \eref{polar} and proceed as in the proof of Theorem \ref{thm2} to obtain the same result. Continuing, expand the function $\psi(\tau)=\pi - \arccos\tau$ for $\tau\in[-1,1]$, $\arccos\tau \in [0,\pi]$ into a power series around 0 and get $\psi(\tau)=\pi/2+\tau+\tau^3/12\pi+\bigo(\tau^5)$. 
	We thus have 
	\begin{align*}
	\E |\grad F_\mu|^2 = \iint_{\Rn\times\Rn} \bigg(\frac{x\cdot y}{4}  +  \frac{(x\cdot y)^2}{2\pi|x||y|}  +  \frac{(x\cdot y)^4}{12\pi|x|^3|y|^3}  +  \bigo\bigg(\frac{(x\cdot y)^6}{|x|^5|y|^5}\bigg)\bigg) d\mu\otimes \mu.
	\end{align*}
	As for the other component, observe that $\grad F_\mu(\theta)\cdot\theta=F(\theta)$ and as was calculated in the proof of Theorem \ref{thm2},
	\begin{align*}
	\E F_\mu^2  =  \iint_{\Rn\times\Rn} \bigg(\frac{|x||y|}{2\pi n}  +  \frac{x\cdot y}{4n}  +  \frac{(x\cdot y)^2}{4\pi n |x||y|}  +  \bigo\bigg(\frac{(x\cdot y)^4}{ n|x|^3|y|^3}\bigg) d\mu\otimes\mu.
	\end{align*}
	The linear component again vanished because $\mu$ is centered, and as $\Cov_1(\mu)=\alpha/\sqrt{n}\cdot\Id$ we have
	\begin{align*}
	\int |x| \ d\mu \ = \ \alpha\sqrt{n}\tab;\tab\iint \frac{(x\cdot y)^2}{|x||y|} d\mu\otimes \mu \ = \ \alpha^2.
	\end{align*}
	According to the additional assumptions of Theorem \ref{thm2},
	\begin{align}
	\iint \frac{(x\cdot y)^4}{|x|^3|y|^3} \ d\mu\otimes\mu  \ - \ 3\iint \frac{(x\cdot y)^2}{n|x||y|} \ d\mu\otimes\mu\ \leq \ \frac{\gamma\alpha^2}{n^2}
	\end{align}
	and in conclusion $\E |\grad_S F_\mu|^2 \ \leq \ C(1+\gamma+\delta)\alpha^2/n^2$, as we required.
	To prove the first assertion of Proposition \ref{gradprop}, simply repeat the argument above while taking the two series expansions one order less.
\end{proof}

\subsection{}\label{app.2deriv}

We prove the following.
\begin{prop}\label{app.2deriv.prop}
	Let $\mu$ be an \Loneiso\ probability measure on $\Rn$, and write $\Cov_1(\mu)=\alpha/\sqrt{n}\cdot\Id$. Assume that $F_\mu$ is $C^2$-smooth and that 
	\begin{align*}
	\intSn \iint_{x,y\in\theta^\perp} (x\cdot y)^2 \I_{|x\cdot y|>|x||y|/2}\ \rho(x) \ \rho(y) \ dx\ dy\ d\sign(\theta) \ \leq \ \frac{\gamma\alpha^2}{n},
	\end{align*}
	where $\rho:\Rn\to\Rplus$ is the density of $\mu$. Assume in addition that
	\begin{equation*}
	\iint \frac{(x\cdot y)^4}{|x|^3|y|^3} \ d\mu\otimes \mu\ \leq\ \frac{\beta\alpha^2}{n} \tab\text{and}\tab \iint \frac{(x\cdot y)^6}{|x|^5|y|^5} \ d\mu\otimes\mu\ \leq\ \frac{\delta\alpha^2}{n^2}.
	\end{equation*}
	Then
	\begin{align*}
	\E\hsnorm{F''_S}^2 \ \leq \ \frac{\alpha^2}{4\pi}(\beta-3) \ + \ \frac{C(1+\beta+\gamma+\delta)\alpha^2}{n}
	\end{align*}
	where $C>0$ is a universal constant.
\end{prop}
The assumption that $\mu$ is an \Loneiso\ log-concave measure is sufficient for Proposition \ref{app.2deriv.prop} to apply with parameters all of the order of magnitude of a universal constant, see \cite[Chapter 6]{thesis} for details.
The essence of the proof will be the calculation of $\E\hsnorm{F_\mu''}^2$. The transition between the spherical and Euclidean second derivatives is given by the following general relation.
\begin{claim}\label{int2dexpansion}
	Let $f:\Sn\to\R$ be $C^2$-smooth and 1-homogeneous. Then
	\begin{equation*}
	\E\hsnorm{f_S''}^2 \ = \ \E \hsnorm{f''}^2 - (n-1)(\E f)^2 - (n-1)\Var(f) +2\E |\grad_S f|^2.
	\end{equation*}
\end{claim}
\begin{proof}
	First, we establish that 
	\begin{equation}\label{fdp}
	\hsnorm{f_S''}^2=\hsnorm{f''}^2-(n-1)f^2-2f\Delta_S f
	\end{equation} 
	pointwise. Indeed, fix $\theta\in\Sn$ and choose an orthonormal basis $e_1,...,e_n$ such that $e_1=\theta$. Recall the formula \eref{deriv2def} for the spherical second derivative, $f''_S(\theta) \ = \ \Pthperp(f''(\theta)-(\grad f(\theta)\cdot\theta)\Id)\Pthperp$. As $f''(\theta)\theta=0$ for any 1-homogeneous function and also $\grad f(\theta)\cdot\theta=f(\theta)$, then equation \eref{fdp} follows from simply taking the elements of $f''_S$ squared, and replacing the Euclidean Laplacian with the spherical one according the formula $\Delta f(\theta)=\Delta_Sf(\theta)+(n-1)(\grad f(\theta)\cdot\theta)+f''(\theta)\theta\cdot\theta$ (see \cite{BCG}). Moving on, on the sphere we have the spectral relation $\E f\Delta_S f=-\E|\grad_S f|^2$ and the Claim is obtained by writing $\Var(f)=\E f^2 - (\E f)^2$.
\end{proof}

At this point note that by Theorem \ref{thm1} and Proposition \ref{gradprop} we have
\begin{equation*}
\Var(F) \ = \ \bigo\bigg(\frac{(1+\beta)\alpha^2}{n^2}\bigg), \tab \E|\grad_S F|^2 \ = \ \bigo\bigg(\frac{(1+\beta)\alpha^2}{n}\bigg)
\end{equation*}
and by a calculation in the proof of Theorem \ref{thm2} also
\begin{equation*}
(\E F)^2 \ = \ \frac{\alpha^2}{2\pi} \ + \ \frac{\alpha^2}{4\pi n} \ + \ \bigo\bigg(\frac{\alpha^2}{n^2}\bigg).
\end{equation*}
Hence already 
\begin{equation}\label{app.2deriv.mid}
\E \|F''_S\|_{\text{HS}}^2 \ = \ \E \|F''\|_{\text{HS}}^2 \ - \ \frac{n\alpha^2}{2\pi} \ + \ \frac{\alpha^2}{4\pi} \ + \ \bigo\bigg(\frac{(1+\beta)\alpha^2}{n}\bigg).
\end{equation}
We thus turn to calculating $\E\hsnorm{F''}^2$ and state a lemma toward this end, sampling points on the sphere by sampling points in lower dimension. Denote by $\lambda_n$ the unique Haar measure on the orthogonal group $O_n$, and by $\Snzero=\{\theta\in\Sn: \theta_1=0\}$ the embedding of $S^{n-2}$ into the hyperplane $\{x\in \Rn: x_1=0\}$.
\begin{lemma}\label{lemmarand}
	Define $T:\On\times\Snzero\times\Snzero\to\Sn\times\Sn$ by
	\begin{equation*}
	T(U,\theta_1,\theta_2)\ = \ (U\theta_1, U\theta_2),
	\end{equation*}
	and write $\nu=\lambda_n\times\sigma_{n-2}\times\sigma_{n-2}$. Then $T_*\nu$ is absolutely continuous with respect to $\sign\times\sign$ and moreover
	\begin{equation*}
	\frac{dT_*\nu}{d(\sign\times\sign)}(\theta_1,\theta_2) \ = \ \frac{\Omega_n}{\sqrt{1-(\theta_1\cdot\theta_2)^2}},
	\end{equation*}
	with the normalizing constant  $\Omega_n=\frac{n-2}{2}\big(\Gamma\big(\frac{n-1}{2}\big)/\Gamma\big(\frac{n}{2}\big)\big)^2$.
\end{lemma}
Indeed, it is easy to be convinced that $T_*\nu$ is invariant under joint rotation, then the density must be a function of $\theta_1\cdot\theta_2$ and can be calculated using a test function, see \cite[Chapter 6]{thesis} for details. Moving on, the proof of Proposition \ref{app.2deriv.prop} is completed with the following calculation. 
\begin{proof}[Proof of Proposition \ref{app.2deriv.prop}]
	The formula \eref{2dcalc} for the Euclidean second derivative yields 
	\begin{align*}\label{int2dcalc}
	\E\ \hsnorm{F''_S}^2 \ &= \ \intSn \int_{x\in\theta^\perp}\int_{y\in\theta^\perp} (x\cdot y)^2 \ \rho(x) \rho(y) \ dx \ dy \ d\sign(\theta).
	\end{align*}
	We drop the region where $|x\cdot y|>|x||y|/2$, according to the assumption of Proposition \ref{app.2deriv.prop}, and the above equals up to $\gamma\alpha^2/n$ to
	\begin{align*}
		 \intSn \int_{x\in\theta^\perp}\int_{y\in\theta^\perp} (x\cdot y)^2\I_{|x\cdot y|\leq |x||y|/2} \ \rho(x) \rho(y) \ dx \ dy \ d\sign(\theta)
	\end{align*}
	It is clear that fixing a point on the sphere and applying a random rotation to it according to $\lambda_n$ amounts to the same as randomizing a point on the sphere according to $\sign$. We thus continue the calculation by
	\begin{align*}
	=\ \intOn \iint_{x,y\in e_1^\perp} (Ux \cdot Uy)^2\I_{|Ux\cdot Uy|\leq |Ux||Uy|/2} \ \rho(Ux)\rho(Uy) \ dx \ dy \ d\lambda_n(U)
	\end{align*}
	Using polar integration for both the integrals over $e_1^\perp$ and writing $\kappa_n$ for the volume of the unit ball in dimension $n$ we get
	\begin{align*}
	= \ (n-1)^2\kappa_{n-1}^2 \intOn &\iint_{r_1,r_2=0}^\infty \iint_{\theta_1,\theta_2\in\Snzero} r_1^nr_2^n\ (U\theta_1\cdot U\theta_2)^2 \I_{|U\theta_1\cdot U\theta_2|\leq1/2}
	\\
	& \rho(r_1U\theta_1)\rho(r_2U\theta_2) \ dr_1 \ dr_2 \ d\sigma_{n-2}(\theta_1) \ d\sigma_{n-2}(\theta_2) \ d\lambda_n(U)
	\end{align*}
	We interchange integration according to $\lambda_n\times\sigma_{n-2}\times\sigma_{n-2}$ with integration according to $\sign\times\sign$ as instructed by Lemma \ref{lemmarand},
	\begin{align*}
	= \ (n-1)^2\kappa_{n-1}^2&\iint_{r_1,r_2=0}^\infty  \iint_{\theta_1,\theta_2\in\Sn} r_1^nr_2^n\  (\theta_1\cdot\theta_2)^2\ \I_{|\theta_1\cdot \theta_2|\leq1/2}
	\\&
	\frac{\Omega_n}{\sqrt{1-(\theta_1\cdot\theta_2)^2}} \ \rho(r_1\theta_1)\rho(r_2\theta_2) \ dr_1 \ dr_2 \ d\sign(\theta_1)\ d\sign(\theta_2)
	\end{align*}
	Now use reverse polar integration twice on $\Rn$ to get
	\begin{align*}
	=\ \frac{\Omega_n(n-1)^2\kappa_{n-1}^2}{n^2\kappa_n^2}\iint_{\Rn\times\Rn}\frac{(x\cdot y)^2}{|x||y|} \ \frac{\I_{|x\cdot y|\leq|x||y|/2}}{\sqrt{1-(x\cdot y)^2/|x|^2|y|^2}} \ d\mu(x) \ d\mu(y).\ \ \ 
	\end{align*}
	The coefficient outside the integral is just $(n-2)/2\pi$, and next move is to apply the expansion $1/\sqrt{1-t^2} = 1+t^2/2+\bigo(t^4)$ valid whenever $|t|<1/2$. We arrive at
	\begin{align*}
	\E\hsnorm{F''}^2 \ &\leq \ \frac{\gamma\alpha^2}{n} \ + \ \frac{n-2}{2\pi}\iint\bigg(\frac{(x\cdot y)^2}{|x||y|} + \frac{(x\cdot y)^4}{2|x|^3|y|^3} +  \bigo\bigg(\frac{(x \cdot y)^6}{|x|^5|y|^5}\bigg)\bigg)\ d\mu\otimes\mu
	\\&
	= \ \frac{\gamma\alpha^2}{n} \ + \ \frac{n-2}{2\pi}\bigg(\alpha^2 \ + \ \frac{\beta\alpha^2}{2n} \ + \ \bigo\bigg(\frac{\delta\alpha^2}{n^2}\bigg)\bigg)
	\\&
	= \ \frac{n\alpha^2}{2\pi} \ - \ \frac{\alpha^2}{\pi} \ + \ \frac{\beta\alpha^2}{4\pi} \ + \ \bigo\bigg(\frac{(\beta+\gamma+\delta)\alpha^2}{n}\bigg).
	\end{align*}

	Combining with equality \eref{app.2deriv.mid} we finally have the desired
	\begin{align*}
	\E\|F''_S\|_\text{HS}^2 \ \leq \ \frac{\alpha^2}{4\pi}(\beta-3) \ + \ \bigo\bigg(\frac{(1+\beta+\gamma+\delta)\alpha^2}{n}\bigg).
	\end{align*}
\end{proof}

\subsection{}\label{app.proximity}

Lemma \ref{proximity} will be proven along the next four claims. The heart of the proof is to obtain a one-sided bound on $Z_{p,\mu}$ that does not involve $T$.
\begin{claim}\label{app.proximity.claim1}
	Under the assumptions of Lemma \ref{proximity}, $Z_{p,\mu}\sgeq n^{p/2-1}$.
\end{claim}
\begin{proof}
	Without loss of generality, let $a>0$ be such that $aT\in SL_n$. We apply Jensen's inequality, Corollary \ref{cor.uniqueness} and inequality \eref{holder} repeatedly. First,
	\begin{align*}
	Z_{p,\mu} \ = \ \int |x|^{p-2} \ d\mu \ &\geq \ \bigg(\int |x|^p \ d\mu\bigg)^{1-2/p} \ 
	\\&
	\geq \ a^{p-2}\bigg(\int |x|^p \ dT_*\mu\bigg)^{1-2/p} \ \simeq \ a^{p-2}n^{p/2-1}
	\end{align*}
	as $T_*\mu$ is isotropic. Second, as $\Tr\Covp(\mu)=nZ_{p,\mu}$ we have
	\begin{align*}
	nZ_{p,\mu} \ = \ \int |x|^p \ d\mu \ \leq \ a^p\int |x|^p \ dT_*\mu \ \simeq \ a^pn^{p/2}.\tab\tab\tab\tab\ 
	\end{align*}
	Hence $a^{p-2}\sleq a^p$ and $a\sgeq 1$. Finally,
	\begin{align*}
	\ nZ_{p,\mu} \ = \ \int |x|^p \ d\mu \ &\simeq \ \bigg(\int |x|^2 \ d\mu \bigg)^{p/2} \ 
	\\&
	\geq \ a^p\bigg(\int |x|^2 \ dT_*\mu\bigg)^{p/2} \ = \ a^pn^{p/2} \ \sgeq \ n^{p/2}.
	\end{align*}
\end{proof}
Corollary \ref{cor.uniqueness} is not necessary to prove Claim \ref{app.proximity.claim1}, see \cite[Chapter 4]{thesis} for an alternative proof.
We may now obtain the first part of Lemma \ref{proximity}.
\begin{claim}\label{app.proximity.claim2}
	Under the assumptions of Lemma \ref{proximity}, $\|T\|_{\text{op}}\sleq 1$.
\end{claim}
\begin{proof}
Fix $\theta\in\Sn$. As $\Tr\Covp(\mu)=Z_{p,\mu}\Id$ and by the Cauchy-Schwartz inequality we have
\begin{align*}
Z_{p,\mu} \ = \ \int (x\cdot \theta)^2 \ |x|^{p-2} \ d\mu(x) \ \leq \ \sqrt{\int (x\cdot\theta)^4 \ d\mu(x)}\cdot\sqrt{\int |x|^{2(p-2)} \ d\mu}.
\end{align*}
We tend to the first multiplier above. By inequality \eref{holder}, a change of variable then by isotropicity of $T_*\mu$,
\begin{align*}
\sqrt{\int(x\cdot\theta)^4 \ d\mu(x)} \ \simeq \ \int (x\cdot\theta)^2 \ d\mu(x) \ = \ \int (x\cdot T^{-1,*}\theta)^2 \ dT_*\mu(x) \ = \ |T^{-1,*}\theta|^2.
\end{align*}
As for the second multiplier, by Jensen's inequality then inequality \eref{paouris} 
\begin{align*}
\sqrt{\int |x|^{2(p-2)} \ d\mu} \ &\leq \ \bigg(\int |x|^{-2} \ d\mu\bigg)^{1-p/2} \ \simeq \ \bigg(\int |x|^2 \ d\mu \bigg)^{p/2-1}\ 
\\&
= \ \bigg(\int |T^{-1}x|^2 \ dT_*\mu \bigg)^{p/2-1} \ 
\\&
\leq \ \|T\|_{op}^{2-p}\bigg(\int |x|^2 \ dT_*\mu\bigg)^{p/2-1} \ = \ n^{p/2-1}\|T\|_{\text{op}}^{2-p} \tab\tab
\end{align*}
as $\|T\|_{\text{op}}^{-1}\leq |T^{-1}x|/|x|\leq \|T^{-1}\|_{op}$ for any $x\in\Rn$. Combining the above with Claim \ref{app.proximity.claim1} we arrive at
\begin{align*}
n^{p/2-1}\|T\|_{\text{op}}^{2-p}\ |T^{-1,*}\theta|^2 \ \sgeq \ Z_{p,\mu} \ \sgeq \ n^{p/2-1}.
\end{align*}
Taking the infimum over all $\theta\in\Sn$ we get
\begin{align*}
\|T\|_{\text{op}}^{p/2-1} \ \sleq \ \inf_{\theta\in\Sn}|T^{-1,*}\theta| \ = \ \|T^*\|_\text{op}^{-1} \ = \ \|T\|_\text{op}^{-1},
\end{align*}
and hence $\|T\|_\text{op} \sleq 1$ and the proof is complete.
\end{proof}
\nline
At this point, note that a one sided bound on $Z_{p,\mu}$ in the opposite direction may be obtained as well. Indeed, 
\begin{align*}
Z_{p,\mu}  \ = \ \int |T^{-1}x|^{p-2} \ dT_*\mu \ \leq \ \|T\|_\text{op}^{2-p} \int |x|^{p-2} \ dT_*\mu \ = \ \|T\|_\text{op}^{2-p}\ Z_{p,T_*\mu},
\end{align*}
while $\|T\|_\text{op}^{2-p}\sleq 1$ by Claim \ref{app.proximity.claim2}. As for the isotropic $T_*\mu$ we have by Jensen's inequality, then inequality \eref{paouris} that
\begin{align*}
Z_{p,T_*\mu} \ \sleq \ \bigg(\int |x|^{-2} \ dT_*\mu\bigg)^{1-p/2} \ \simeq \ \bigg(\int |x|^2 \ dT_*\mu \bigg)^{p/2-1} \ = \ n^{p/2-1}.
\end{align*}
We arrive at $Z_{p,\mu}\simeq n^{p/2-1}$ and continue with the proof of Lemma \ref{proximity}.
\begin{claim}\label{app.proximity.claim3}
	Under the assumptions of Lemma \ref{proximity}, for any $\theta\in\Sn$
	\begin{align*}
	\intRn (x\cdot\theta)^2 \ |x|^{p-2} \ dT_*\mu(x) \ \sgeq \ n^{p/2-1}.
	\end{align*}
\end{claim}
\begin{proof}
	The proof is based on the following simple probabilistic inequality, obtained by the Cauchy-Schwartz inequality. If $X$ is a non-negative random variable and $A$ is an event with $\P(A)\geq 1-(\E X)^2/4\E X^2$, then $\E X\I_A\geq \E X/2$.
	By inequality \eref{holder}, let $M>0$ be a universal constant such that 
	\begin{align*}
	\intRn (x\cdot\theta)^4 \ dT_*\mu(x) \ \leq \ M\bigg(\intRn (x\cdot\theta)^2 \ dT_*\mu(x)\bigg)^2 \ = \ M
	\end{align*}
	for all $\theta\in\Sn$. By Chebyshev's inequality,
	\begin{align*}
	T_*\mu\big(\big\{\ |x|>\sqrt{4Mn}\ \big\}\big) \ \leq \ \oo{4Mn}\intRn |x|^2 \ dT_*\mu \ = \ \oo{4M}.
	\end{align*}
	We may thus apply the probabilistic inequality to obtain for any $\theta\in\Sn$,
	\begin{align*}
	\int_{|x|\geq\sqrt{4Mn}} (x\cdot\theta)^2 \ dT_*\mu(x) \ \geq \ \half\intRn (x\cdot\theta)^2 \ dT_*\mu(x) \ = \ \half.
	\end{align*}
	To conclude the proof, let $\theta\in\Sn$ and see that
	\begin{align*}
	\intRn (x\cdot\theta)^2\  |x|^{p-2} \ dT_*\mu(x) \ &\geq \ \int_{|x|\leq\sqrt{4Mn}}(x\cdot\theta)^2 \ |x|^{p-2}\ dT_*\mu(x)
	\\&
	\geq \ (4Mn)^{p/2-1} \int_{|x|\leq\sqrt{4Mn}}(x\cdot\theta)^2 \ dT_*\mu(x) \ 
	\\&
	\sgeq \ n^{p/2-1}.
	\end{align*}
\end{proof}

\begin{claim}\label{app.proximity.claim4}
	Under the assumptions of Lemma \ref{proximity}, $\|T^{-1}\|_{\text{op}}\sleq 1$.
\end{claim}
\begin{proof}
	 Fix $\theta\in\Sn$. As $\Covp(\mu)=Z_{p,\mu}\Id$ and by Claim \ref{app.proximity.claim3} we have
	 \begin{align*}
	 Z_{p,\mu} \ &= \ \int (x\cdot\theta)^2 \ |x|^{p-2} \ d\mu(x) \ 
	 \\&
	 = \ \int (x\cdot T^{-1,*}\theta)^2 \ |T^{-1}x|^{p-2} \ dT_*\mu
	 \\&
	 \geq \ \|T^{-1}\|_\text{op}^{p-2}\int (x\cdot T^{-1,*}\theta)^2 \ |x|^{p-2} \ dT_*\mu  \ 
	 \\&
	 \sgeq\ \|T^{-1}\|_\text{op}^{p-2}\ |T^{-1,*}\theta|^2\ n^{p/2-1}.
	 \end{align*}
	 Taking the supremum over all $\theta\in\Sn$ we get
	 \begin{align*}
	 Z_{p,\mu}\ n^{1-p/2}\ \|T^{-1}\|_\text{op}^{2-p} \ \sgeq \ \sup_{\theta\in\Sn }|T^{-1,*}\theta|^2 \ = \ \|T^{-1,*}\|_\text{op}^2 \ = \ \|T^{-1}\|_\text{op}^2
\end{align*}
	and recall that $Z_{p,\mu}\simeq n^{p/2-1}$.
\end{proof}


\begin{thebibliography}{99}

	\bibitem {BCG} Bobkov S. G., Chistyakov G. P., Götze F., Second order concentration on the sphere.
	{\it Communications in Contemporary Mathematics. Online 13 September 2016}
	
	\bibitem {BLM} Boucheron S., Lugosi G., Massart P., Concentration Inequalities; a nonasymptotic theory of independence.
	{\it Oxford University Press, 2013. ISBN: 9780199535255}
	
	\bibitem {BGVV} Brazitikos S., Giannopoulos A., Valettas P., Vritsiou B. H., Geometry of isotropic convex bodies.
	{\it Mathematical Surveys and Monographs, 196. American Mathematical Society, Providence, RI, 2014. xx+594 pp. ISBN: 978-1-4704-1456-6.}
	
	\bibitem {thesis} Buchweitz E., M.Sc. thesis.
	{\it Tel Aviv University library, http://primage.tau.ac.il/libraries/theses/exeng/free/2963618.pdf.}
	
	\bibitem {Ch} Chatterjee S., Superconcentration and related topics.
	{\it Springer Monographs in Mathematics. Springer International Publishing, 2014. ISBN: 978-3-319-03885-8.}
	
	\bibitem {EK} Eldan R., Klartag B., Approximately Gaussian marginals and the hyperplane conjecture. 
	{\it Concentration, functional inequalities and isoperimetry, 55-68,
		Contemp. Math., 545, Amer. Math. Soc., Providence, RI, 2011. }
		
	\bibitem{GM} Giannoupoulos A., Milman V. D., Extremal problems and isotropic positions of convex bodies.
	{\it  Israel J. Math. 117 (2000), 29-60.}	
		
	\bibitem {G} Gromov M., Paul Lévy's isoperimetric inequality.
	{\it reprinted in `Metric structures for Riemannian and non-Riemanninan spaces', Birkhäuser Boston, 2007. ISBN: 978-0-8176-4582-3.}
	
	\bibitem {H} Haagerup U., The best constants in the Khintchine inequality.
	{\it  Studia Math. 70 (1981), no. 3, 231-283 (1982).}
	
	\bibitem {L} Ledoux M., The concentration of measure phenomenon.
	{\it Mathematical Surveys and Monographs, 89. American Mathematical Society, 2005. ISBN: 978-0-8218-3792-4.}
		
	\bibitem {PV} Paouris G., Valettas P., A small deviation inequality for convex functions (Preprint).
	
	\bibitem {S} Schmuckenschläger M., Bernstein inequalities for a class of random variables.
	{\it  Proc. Amer. Math. Soc. 117 (1993), no. 4, 1159-1163.}
	
	\bibitem {TE} Tricomi F., Erdélyi A., The asymptotic expansion of a ratio of Gamma functions.
	{\it Pacific J. Math. 1, (1951). 133-142. }
	
\end{thebibliography}
\end{document}